\documentclass{article}
\usepackage{amsmath}
\usepackage{amsfonts}
\usepackage{enumerate}
\usepackage{color}
\usepackage{graphicx}
\usepackage{hyperref} %
\usepackage{amssymb}
\usepackage{amscd}

\newtheorem{teorema}{Theorem}
\newtheorem{lema}{Lemma}
\newenvironment{proof}[1][Proof]{\noindent\textbf{#1. }}{\hfill\rule{0.5em}{0.5em}}

\begin{document}

\title{A multiscale model of cell mobility: From a kinetic to a hydrodynamic description}

\author{J. Nieto \& L. Urrutia\thanks{University of Granada. Departamento de Matem\'atica
Aplicada,  18071 Granada, Spain. e-mail: jjmnieto@ugr.es,
lurrutia@ugr.es}}
\date{}
\maketitle

\begin{abstract}
This paper concerns a  model for tumor cell migration through the surrounding extracellular matrix by considering mass balance  phenomena involving the chemical interactions produced on the cell surface. The well--posedness  of this model is proven. An asymptotic analysis via a suitable hydrodynamic limit completes the des\-crip\-tion of the macroscopic behaviour.
\end{abstract}



\section{Introduction}

There is a huge literature describing mathematical models for cell migration through the extracellular matrix (ECM), specially tumor cells, since they u\-sua\-lly try to reach a blood vessel to obtain nutrients or simply invade other parts of the body in a metastatic process. There are a lot of biological mechanisms involved in cell movement such as signaling, diffusion, chemotaxis, haptotaxis, reorientation due to the surrounding tissue fibers, cell--cell interactions, etc., and also some mechanical considerations as balance laws, mechanical forces, pressure, etc. (see for example \cite{BBNS2,BBTY,FLP,KelSur, Hillen,Maini,Mallet,OD}).

In general, there is an analogy with the models for mechanical particles, where biological considerations are included in several ways. For example, reorientations of the particles due to biological interactions can be modeled by a Boltzmann--type equation where the usual collision kernel plays the role of a reorientation kernel.
Of course, macroscopic descriptions (Navier--Stokes or Keller--Segel models) are very common, and the connections between kinetic and hydrodynamic models by means of limiting procedures have been largely treated in the literature (see for example \cite{BBNS3,BBNS1,BBNS2,BBNS4,BBTY,FLP,Hillen}).

Following the analogy with mechanical models, it is remarkable the framework of Kinetic Theory of Active Particles (KTAP) introduced by Bellomo et al. (see for example \cite{BBNS3} and references therein), where \textit{active particles} play the double role of mechanical entities and living beings. This theory allows to cons\-truct models for cell movement that take into account the he\-te\-ro\-geneity of cells, the biological interactions, birth/death phenomena, and also different scales of description.
In this spirit, a recent paper by Kelkel and Surulescu \cite{KelSur}, presents a multiscale model describing the evolution of a tumor cell population density where the movement of the cells is mainly due to receptor dynamics on the cell surface. The model links several processes such as haptotaxis,
 binding of the cell surface to the ECM fibers, chemotaxis due to a substance originated from the degradation of tissue fibers, and the law of mass action of the receptor on the cell surface.

In this work, we start from the multiscale model presented in \cite{KelSur}, and include some mechanical and biological
considerations that improve it. Actually, in the equations for the ECM dynamics we introduce the mass balance due to interactions with the cell population.
This mass exchange, together with the creation and degradation of
substances, constitute a key part in the state of the ECM, and so modifies the dynamics of the population.
Moreover, we perform a hydrodynamic limit which provides macroscopic information on the
behaviour of the cell population and preserves the influence of the two main biological processes, haptotaxis and chemotaxis.

For the sake of selfconsistency, we briefly describe in the next subsection the elements involved in cell motion as well as our improvements. In Section 2, we prove  existence and uniqueness of solution for the obtained model, and in Section 3 we perform the high--field limit. In particular, we will obtain closed relations between the averaged chemical substances involved in cell movement and the respective concentration in the ECM.
\subsection{The multiscale model}

Concerning the two processes introduced before, haptotaxis and chemotaxis, we find two different chemical compounds in the ECM (see \cite{KelSur} for details), each one related to one type of cell--environment interaction: An o\-rien\-ted protein fiber, responsible of haptotaxis, and a  chemical compound coming from de\-ge\-ne\-ra\-tion of the aforesaid fibers, responsible for chemotaxis.
We denote $Q(t,x,\theta )$ the density of protein fibers at time $t$ and position $x$, oriented towards $\theta \in \mathbb{S}^{n-1}$ for some $n\geq 1$. The density of protein fibers at time $t$ and position $x$ is denoted by $\bar{Q}(t,x)$:
$$ \bar{Q}(t,x) := \int_{\mathbb{S}^{n-1}} Q(t,x,\theta ) d\theta. $$
Finally, denote $L(t,x)$ the concentration of the other chemical compound, a proteolytic product coming from degradation of ECM fibers.
From now on, we will use the same notation for the compounds and for their densities and concentrations. We will call them the $\bar{Q}$ and $L$ compounds, respectively.

The final model is a system consisting of a kinetic model for the cell po\-pu\-la\-tion (stemming from KTAP) and two macroscopic reaction and reaction--diffusion equations for the chemical compounds. The cell population will be treated as a system of active particles, meanwhile macroscopic models are used for the chemicals. At this point, the aforesaid improvements to the model in \cite{KelSur} are introduced, including a reaction term which takes into account the balance mass of the compounds due to the chemical reactions produced in the cell surface.

We describe the cell population by means of a standard distribution function $f(t,x,v,y)$ depending on time $t$, space $x$, velocity $v$ and \textit{activity} $y$ (which will be described below), verifying the following equation deduced in \cite{KelSur},
\begin{equation} \frac{\partial f}{\partial t} + v\cdot \nabla _x f + \nabla _y \cdot (G(y, \bar{Q}, L)f)= \mathcal{H}(f,Q)+ \mathcal{L}(f) + \mathcal{C}(f,L),
\label{eqnf}
\end{equation}
where the right--hand side models the cell mobility by way of velocity changes and the $y$--divergence term is related to the cell membrane reactions.
Concretely,
\begin{itemize}
\item the term $\mathcal{H}$, modeling haptotaxis, is
\begin{align*}
\hspace{-0,6 cm} \mathcal{H} (f,Q)(t,x,v,y)& := \! \int_V \int_{\mathbb{S}^{n-1}} p_h(t,x,v',y)\psi(v; v', \theta)f(t,x,v',y)Q(t,x,\theta ) d\theta dv' \\
& - p_h(t,x,v,y)\ f(t,x,v,y) \int_V \int_{\mathbb{S}^{n-1}} \psi(v'; v, \theta)Q(t,x,\theta ) d\theta dv'  ;
\end{align*}
\item the turning operator $\mathcal{L}$ models random changes in velocity,
\begin{align*}
\mathcal{L}(f)(t,x,v,y) & := \int_V p_l(t,x,v',y)\alpha_1(y)T(v,v')f(t,x,v',y)dv' \\
& - \ p_l(t,x,v,y)\alpha_1(y)f(t,x,v,y)\int_V T(v',v) dv';
\end{align*}
\item and the chemotactic term, $\mathcal{C}$, reads
\begin{align*}
\mathcal{C} (f,L)(t,x,v,y) & := \int_V p_c(t,x,v',y)\alpha_2(y)K[\nabla L](v,v')f(t,x,v',y)dv' \\
& - \ p_c(t,x,v,y)\alpha_2(y)f(t,x,v,y) \int_V K[\nabla L](v',v) dv'.
\end{align*}
\end{itemize}
Here $p_h, p_l$ and $p_c$ are the interaction frequencies, $\psi $, $T$ and $K$ are the interaction kernels, and $\alpha_i$ are nonnegative weight functions satisfying $\alpha_1+\alpha_2=1$.

In order to define the activity $y$ and the cell membrane reaction terms, we need to recover the law of mass action of the reactions produced at the cell membrane involving the  two chemicals in the ECM and the receptors on the cell,
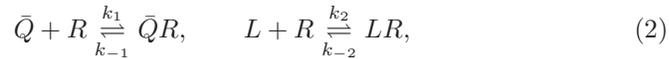
\begin{equation}
\bar{Q} + R  \overset{k_1}{\underset{k_{-1}}{\rightleftharpoons} }\bar{Q}R,
\qquad L + R
\overset{k_2}{\underset{k_{-2}}{\rightleftharpoons}} LR, \label{LAM} \end{equation}
where $R$ stands for the free enzyme on the cell surface and $\bar{Q}R$ and $LR$ represent the respective complexes once the enzyme binds the ECM chemical. Then, $y$ is defined as the two--component vector of microscopic concentrations of the two cell--membrane compounds $\bar{Q}R$ and $LR$, respectively. It is defined in the set
$$Y=\{(y_1,y_2) \in (0,R_0)\times (0,R_0) \ : \  y_1 +y_2 <R_0\},$$
where $R_0 > 0$ represents the maximum concentration of receptors on the cell surface. The function $G$ is given by the expression
\begin{equation*}
G(y,q,l):=\begin{pmatrix}
k_1(R_0-y_1-y_2)q-k_{-1}y_1 \\ k_2(R_0-y_1-y_2)l-k_{-2}y_2  \end{pmatrix},\label{G}
\end{equation*}
whose rows represent the equations associated with (\ref{LAM}).

Now, we introduce the macroscopic equations for the free chemicals $Q$ and $L$ in the ECM:
\begin{align}
\frac{\partial Q}{\partial t} = & -\kappa \left( \int _V \int _Y \left( 1 - \left| \theta \cdot \frac{v}{|v|} \right| \right) f dvdy \right) Q \nonumber \\
& -k_1Q\int _V \int _Y (R_0-y_1-y_2)fdvdy + \frac{k_{-1}}{|\mathbb{S}^{n-1}|} \int _V \int _Y y_1fdvdy, \label{eqnQ}
\end{align}
and
\begin{align}
\frac{\partial L}{\partial t} = & \kappa \int _{\mathbb{S}^{n-1}}\left( \int _V \int _Y \left( 1-\left| \theta \cdot \frac{v}{|v|} \right| \right) f dvdy \right) Q d\theta - r_L L + D_L \Delta _x L \nonumber \\
 & -k_2L\int _V \int _Y (R_0-y_1-y_2)fdvdy + k_{-2} \int _V \int _Y y_2fdvdy. \label{eqnL}
\end{align}
We identify here the models deduced in \cite{KelSur,Hillen}, where the first term represents, in both equations, the production of chemical $L$ by degradation of the fiber $Q$ after interaction with a cell. Also, decay and diffusion of chemical $L$ with rate $r_L$ can be observed. Finally, the two last reaction  terms are introduced in this paper, with the aim of adding the mass balance due to the cell membrane interactions, which completes the previous model.

Now, we are interested in the well--posedness of the whole system (\ref{eqnf}), (\ref{eqnQ}), and (\ref{eqnL}) in an adequate space.

\section{Existence and uniqueness of solution}

We start by introducing the sets in which the variables are defined. The space variable can be considered in the whole  $\mathbb{R}^n $ for some $n\geq 1$, meanwhile the activity $y$ is in the set $Y$ described above. Finally, the velocity $v$ will be defined in $V=[s_1, s_2] \times \mathbb{S}^{n-1}$ with $0\leq s_1 < s_2 < \infty$.
Now, we recall some useful estimates that can be found in \cite{KelSur}.

\begin{lema}[Properties of integral operators] Let $T_0>0$ and $0\leq t \leq T_0 $. Then, the following properties hold:
 \label{IntOp}

\begin{enumerate}
\item Let $p_h(t) \in L^{\infty}(\mathbb{R}^n \times V \times Y)$ and  $\psi(v; v', \theta)$ be non-negative functions verifying:
\begin{equation} \int _V \psi (v; v', \theta) dv =1, \ \ \ \ \int _V \psi(v; v', \theta) dv' \leq M . \label{KerH} \end{equation}
Then the integral operator $\mathcal{H}$ is a continuous bilinear mapping from \\ $L^p(\mathbb{R}^n \times V \times Y) \times L^\infty (\mathbb{R}^n \times \mathbb{S}^{n-1})$ to $L^p(\mathbb{R}^n \times V \times Y)$ $(p=1,\infty )$, verifying
$$\| \mathcal{H}(f(t), Q(t)) \| _p \leq C\| p_h(t) \|_\infty \| \bar{Q}(t) \| _\infty \| f(t) \| _p. $$
Furthermore, if $Q(t) \in L^1 (\mathbb{R}^n \times \mathbb{S}^{n-1})$, then
$$\| \mathcal{H}(f(t), Q(t)) \| _1 \leq C\| p_h(t) \|_\infty \| \bar{Q}(t) \| _1 \| f(t) \| _\infty.$$
\item Assume that $0\leq p_l(t) \in L^\infty(\mathbb{R}^n \times V \times Y)$,  $\alpha _1 \in L^\infty(Y)$ and $T(v,v')$ are given functions such that:
\begin{equation} \int _V T(v,v')dv=1; \ \ \ \ |T(\cdot, v)| \leq C|v| . \label{KerL} \end{equation}
Then, the integral operator $\mathcal{L} $ is a continuous mapping from $L^p(\mathbb{R}^n \times V \times Y)$ to $L^p(\mathbb{R}^n \times V \times Y)$ $(p=1,\infty )$, and the inequality
$$ \| \mathcal{L}(f(t)) \| _p \leq 2\| p_l (t) \| _\infty \| f(t) \| _p $$
holds.

\item Assume that $0\leq p_c(t) \in L^\infty(\mathbb{R}^n \times V \times Y)$, $\alpha _2 \in L^\infty(Y)$ and $K[F](v,v')$ are given functions such that:
\begin{equation} \int _V K(v,v')dv=1, \ \ \ \ |K(\cdot, v)| \leq C|v|, \ \ \ \ |K[F] - K[G]| \leq C|F-G|. \label{KerC} \end{equation}
Then, the integral operator $\mathcal{C} $ is a continuous mapping from \\ $L^p(\mathbb{R}^n \times V \times Y) \times L^\infty (\mathbb{R}^n)$ to $L^p(\mathbb{R}^n \times V \times Y)$ $(p=1,\infty )$, verifying
$$ \| \mathcal{C}(f(t), L(t)) \| _p \leq 2M\| p_c (t) \| _\infty \| f(t) \| _p. $$
\end{enumerate}
\end{lema}

To carry on the proof of existence and uniqueness of solution to the system (\ref{eqnf})-(\ref{eqnQ})-(\ref{eqnL}), we first add suitable initial conditions to set up a Cauchy problem: $f(t=0)=f_0$,  $Q(t=0)=Q_0$, $L(t=0)=L_0$. We will follow analogous techniques to those developed in \cite{KelSur}.

The first step is uncoupling the equations, replacing with a given non-negative function on each equation:
\begin{align}
f^* \in & \  L^\infty (0,T_0; L^1(\mathbb{R}^n \times V \times Y) \cap L^\infty (\mathbb{R}^n \times V \times Y) ), \nonumber \\
Q^* \in & \  L^\infty (0,T_0; L^1(\mathbb{R}^n \times \mathbb{S}^{n-1})\cap L^\infty(\mathbb{R}^n \times \mathbb{S}^{n-1})), \label{fest} \\
L^* \in & \  L^\infty (0,T_0; W^{1,1}(\mathbb{R}^n) \cap L^\infty (\mathbb{R}^n) ). \nonumber
\end{align}
The uncoupled system reads:

\begin{equation} \frac{\partial f}{\partial t} + v\cdot \nabla _x f + \nabla _y \cdot (G(y, \bar{Q}^*, L^*) f) = \mathcal{H}(f,Q^*)+ \mathcal{L}(f) + \mathcal{C}(f, L^*) + g,\label{flineal} \end{equation}
\begin{align}
\frac{\partial Q}{\partial t} = & -\kappa \left( \int _V \int _Y \left( 1- \left| \theta \cdot \frac{v}{|v|} \right| \right) f^* dvdy \right) Q \nonumber \\
& -k_1Q\int _V \int _Y (R_0-y_1-y_2)f^*dvdy + \frac{k_{-1}}{|\mathbb{S}^{n-1}|} \int _V \int _Y y_1f^*dvdy + h, \label{Qlineal} \\
\frac{\partial L}{\partial t} = & \, \kappa \int _{\mathbb{S}^{n-1}}\left( \int _V \int _Y \left( \left| 1- \theta \cdot \frac{v}{|v|} \right| \right) f^* dvdy \right) Q^* d\theta - r_L L + D_L \Delta _x L  \nonumber \\
 & -k_2L\int _V \int _Y (R_0-y_1-y_2)f^*dvdy + k_{-2} \int _V \int _Y y_2f^*dvdy, \label{Llineal}
\end{align}
where $g(t,x,v,y)$ and $h(t,x,\theta)$ are two additional functions.

We start with the following result for equation (\ref{flineal}).
\begin{teorema}
(See  \cite[Theorem 1]{KelSur}) \label{existf}
Let $f_0 \in L^1(\mathbb{R}^n \times V \times Y) \cap L^\infty (\mathbb{R}^n \times V \times Y) $ be a non-negative function, $g\in L^1 (0,T_0; (L^1\cap L^\infty)(\mathbb{R}^n \times V \times Y)) $ and consider $Q^\ast$ and $L^\ast$ verifying (\ref{fest}).
Also assume that
\begin{enumerate}[(i)]
\item $f_0 \in L^\infty (\mathbb{R}^n \times V; L^1(Y)) \cap L^\infty (\mathbb{R}^n \times V; W^{1,\infty}(Y)) \cap L^1 (\mathbb{R}^n \times V; W^{1,1}(Y)).$
\item $p_c, p_h, \nabla _yp_h $ and $\nabla _y p_c $ belong to $L^\infty (0,T_0; L^\infty (\mathbb{R}^n \times V \times Y)) $.
\item $\nabla _y \alpha _1, \nabla _y \alpha _2 $ are bounded.
\end{enumerate}
Then, under the hypothesis of Lemma \ref{IntOp}, there exists an unique weak solution $f$ to Eq. (\ref{flineal}) with initial condition $f_0$. Also, $f$ verifies 
\begin{equation}
\|f(t) \|_p \leq \left( \|f_0 \|_p + \int _0 ^{T_0} \|g(\tau ) \|_p d\tau \right) (1+Cte^{Ct})
\label{desf}
\end{equation}
$(p=1, \infty )$, where $C$ is a positive constant depending on $\| Q^*\| _\infty$ and $\| L^* \|_\infty$.
Furthermore, if $g\equiv 0$, then
$$\| \nabla _y f(t) \| _\infty \leq \left(\| \nabla _y f(0) \| _\infty + C\| f_0 \| _\infty (T_0+CT_0^2e^{CT_0}) \right)e^{CT_0}.  $$
\end{teorema}

Now we prove the next result for the equation to $Q$.
\begin{teorema}
 Let $Q_0,\, h(t) \in L^1(\mathbb{R} ^n \times \mathbb{S} ^{n-1} ) \cap L^\infty (\mathbb{R} ^n \times \mathbb{S} ^{n-1} ) $, $Q_0\geq 0$ and $f^*$ as in (\ref{fest}). Then, there exists an unique function $Q(t) \in L^1(\mathbb{R} ^n \times \mathbb{S} ^{n-1} ) \cap L^\infty (\mathbb{R} ^n \times \mathbb{S} ^{n-1} )$ that solves (\ref{Qlineal}) with initial condition $Q(t=0)=Q_0$. Furthermore, $\forall t \in [0,T_0]$ this solution verifies:
 \begin{equation}
 \| Q(t) \|_p \leq  \| Q_0 \| _p + \int _0 ^{T_0} \| h(\tau ) \| _p d\tau+ C\int _0 ^{T_0} \| \rho ^* (\tau )\| _p d\tau,  \label{desQ}
 \end{equation}
 $(p=1, \infty)$, where
 $$ \rho ^* (t,x):=\int_V \int_Y f^*(t,x,v,y) dvdy. $$
Moreover, if $h=0$, then $Q(t) \geq 0 $.
\label{existQ}
\end{teorema}

\begin{proof}
Existence and uniqueness of solution is straightforward, since (\ref{Qlineal}) is a linear differential equation.
Actually, the solution can be written as
\begin{equation}\label{aux99}
Q(t)= e^{\int _0 ^t J(\tau)d\tau } Q_0  + \int _0 ^t e^{\int _s ^t J(\tau)d\tau} \left( \frac{k_{-1}}{|S^{n-1}|}\int_V \int _Y  y_1 f ^* (s) dvdy + h(s) \right)ds,
\end{equation}
where the function $J$ is given by
$$
J(t,x,\theta) := -\kappa \int _V \int _Y \left( 1 - \left| \theta \cdot \frac{v}{|v|}  \right| \right) f^* dvdy  - k_1\int_V \int_Y (R_0-y_1-y_2)f^* dvdy.
$$
Taking now absolute values in (\ref{aux99}), we deduce that
\[
|Q(t) | \leq e^{\int _0 ^t J(\tau)d\tau} |Q_0 | + \int _0 ^t  e^{\int _s ^t J(\tau)d\tau}\left(
\frac{k_{-1}}{|S^{n-1}|}\int_V \int _Y  y_1 f ^* (s) dvdy + | h(s)|\right)  ds.
\]
Finally, using that $J$ is non positive, we find (\ref{desQ}) with $C=k_{-1}R_0 /|\mathbb{S}^{n-1}|$.
\end{proof}
\medskip


The equation for $L$ is a linear perturbation of the heat equation. A classic result leads to
\begin{teorema}
Let $f^*$ and $Q^*$ satisfy (\ref{fest}). Then, there exists an unique non--negative function $L\in L^\infty (\mathbb{R}^n) $ that solves (\ref{Llineal}) with initial condition  $L(t=0)=0$. Furthermore, for $p=1, \infty$, we have
\begin{equation} \| L(t) \|_p \leq C\left( \| f^*(t) \| _\infty \| Q^*(t) \| _p + \| \rho ^*(t) \| _p \right), \label{desL} \end{equation}
$$ \| \nabla L(t) \|_1 \leq C\left( \| f^*(t) \| _\infty \| Q^*(t) \| _1 + \| \rho ^*(t) \| _1 \right). $$
\label{existL}
\end{teorema}

With these results in mind, we can prove the existence and uniqueness of solution to the full model (\ref{eqnf})--(\ref{eqnQ})--(\ref{eqnL}). We introduce the following functional spaces:
\begin{eqnarray*}
\mathbb{X}_f & := & L^\infty (0,T_0; L^1(\mathbb{R}^n \times V \times Y)\cap L^\infty(\mathbb{R}^n \times V \times Y)), \\
\mathbb{X}_Q & := & L^\infty (0,T_0; L^1(\mathbb{R}^n \times \mathbb{S}^{n-1})\cap L^\infty(\mathbb{R}^n \times \mathbb{S}^{n-1})),\\
\mathbb{X}_L & := & L^\infty (0,T_0; W^{1,1}(\mathbb{R}^n) \cap L^\infty (\mathbb{R}^n) ),\\
\mathbb{X} & := & \mathbb{X}_f \times \mathbb{X}_Q \times \mathbb{X}_L,
\end{eqnarray*}
endowed with their natural norms. We can now give the main result of this section.

\begin{teorema}
Let $0\leq Q_0 \in L^1(\mathbb{R} ^n \times \mathbb{S} ^{n-1} ) \cap L^\infty (\mathbb{R} ^n \times \mathbb{S} ^{n-1} ) $  and $f_0$ be in the conditions of Theorem \ref{existf}. Then, given the initial conditions $f(t=0)=f_0$, $Q(t=0)=Q_0$, $L(t=0)=0 $, there exists
$ T_0>0$ such that the system  (\ref{eqnf})--(\ref{eqnQ})--(\ref{eqnL})  has an unique weak solution $(f,Q,L)\in \mathbb{X} $ defined in $[0,T_0]$. \end{teorema}

\begin{proof}
First of all, we construct the sequence $(f^j, Q^j, L^j)$, defined in $[0, T_0]$ as the corresponding solution of (\ref{flineal})--(\ref{Qlineal})--(\ref{Llineal}) with $g,h=0$, initial data $(f_0, Q_0, 0)$, and $(f^\ast, Q^\ast, L^\ast )= (f^{j-1}, Q^{j-1}, L^{j-1})$ for any $j\geq 1$, starting with $f^0=Q^0=L^0=0$.
Using Theorems \ref{existf}, \ref{existQ} and \ref{existL}, it is straightforward to prove that the sequence is well--defined in $\mathbb{X}$. Actually, for a constant $R> 2 \|(f_0,Q_0,0)\|_{\mathbb{X}}$, we can choose  $T_0$  small enough such that the sequence is in the closed ball of radius $R$ (denoted $B(R)$) of the space $\mathbb{X}$.

Our objective is to prove that this sequence converges to a solution of (\ref{eqnf})--(\ref{eqnQ})--(\ref{eqnL}) with initial data $(f_0, Q_0, 0)$. To do that, we study the difference between two consecutive elements of the sequence. First, we note that $(f^{j+1}-f^j)$ satisfies Eq. (\ref{flineal}) with vanishing initial condition, $(Q^\ast, L^\ast)=(Q^{j}, L^{j})$, and
\begin{eqnarray*}
g&:=&\mathcal{H}(f^{j}, Q^{j}-Q^{j-1})  + \mathcal{C}(f^{j},L^{j})-C(f^{j}, L^{j-1}) \\ &&+ \nabla _y \cdot \Big( (G(y,Q^{j-1}, L^{j-1})-G(y,Q^{j},L^{j}))f^{j} \Big).
\end{eqnarray*}
Using Theorem \ref{existf}, Lemma \ref{IntOp}, and the trivial inequality
$$| G(y,q,l) - G(y, \widehat{q}, \hat{l}) | + | \nabla _y \cdot (G(y,q,l) -
G(y, \widehat{q}, \hat{l})) | \leq C(|q-\hat{q}|+|l-\hat{l}|),
$$
we can easily deduce
\begin{eqnarray*}
\| f^{j+1}-f^{j} \| _{L^\infty (0,T_0; L^1(\mathbb{R}^n \times V \times Y))} & \leq & C(R)\Big(\| Q^{j}-Q^{j-1} \| _{L^\infty (0,T_0; L^1(\mathbb{R}^n \times \mathbb{S}^{n-1} ))}\\
& &+ \| L^{j}-L^{j-1} \| _{L^\infty (0,T_0; W^{1,1}(\mathbb{R}^n))}\Big),
\end{eqnarray*}
with $C$ increasingly dependent on $T_0$.

Analogously, for $(Q^{j+1}-Q^{j})$, it can be proven that
\begin{equation*}
\| Q^{j+1}-Q^{j} \| _{L^\infty (0,T_0; L^1(\mathbb{R}^n \times \mathbb{S}^{n-1}))} \leq C(R)\Big(\| f^{j}-f^{j-1} \| _{L^\infty (0,T_0; L^1(\mathbb{R}^n \times V \times Y ))}\Big),
\end{equation*}
by noticing that $(Q^{j+1}-Q^{j})$ solves (\ref{Qlineal}) with $f^\ast=(f^{j}-f^{j-1})$, vanishing initial data, and
\begin{align*}
h := &\ - Q^{j} \int _V \int _Y \left( \kappa \left( 1 - \left| \theta \cdot \frac{v}{|v|} \right| \right) + k_1 (R_0-y_1-y_2)\right) (f^{j}-f^{j-1}) dvdy    \\
 & - (Q^{j+1}-Q^{j}) \int _V \int _Y \left( \kappa \left( 1 - \left| \theta \cdot \frac{v}{|v|} \right| \right) + k_1 (R_0-y_1-y_2)\right)f^{j-1} dvdy.
\end{align*}

{}Finally, for $(L^{j+1}-L^{j})$, a similar argument allows us to find
\begin{eqnarray*}
\| L^{j+1}-L^{j}\| _{L^\infty(0,T_0; W^{1,1}(\mathbb{R}^n))} & \leq & C(R)\Big(\| f^{j}-f^{j-1} \| _{L^\infty (0,T_0; L^1(\mathbb{R}^n \times V \times Y ))} \\
 & & + \| Q^{j}-Q^{j-1} \| _{L^\infty (0,T_0; L^1(\mathbb{R}^n \times \mathbb{S}^{n-1} ))} \Big).\\
\end{eqnarray*}

Analogous estimates can be obtained in $L^\infty$, using the corresponding in\-e\-qua\-li\-ties. We finally get
$$
\| (f^{j+1}-f^{j},Q^{j+1}-Q^{j},L^{j+1}-L^{j}) \| _{\mathbb{X}} \leq C\| (f^{j}-f^{j-1},Q^{j}-Q^{j-1},L^{j}-L^{j-1}) \| _{\mathbb{X}}.
$$
In order to complete the proof, we can choose $T_0$ small enough in such a way that $C<1$, and then conclude the convergence of the sequence to a solution of the system.

To prove the uniqueness, we can consider the operator $\mathcal{T}:\mathbb{X} \rightarrow \mathbb{X}$ as the solution operator of the uncoupled system, i.e., $\mathcal{T}(f^\ast, Q^\ast, L^\ast )$ is the solution of (\ref{flineal})--(\ref{Qlineal})--(\ref{Llineal}) in $[0,T_0]$ with $g,h=0$ and initial data $(f_0, Q_0, 0)$. We note that a fixed point of $\mathcal{T}$ would be a solution in $\mathbb{X}$ of the full system (\ref{eqnf})--(\ref{eqnQ})--(\ref{eqnL}) and vice versa.
{}The previous computations give us the contractivity of $\mathcal{T}$ for $T_0$ small enough, and using Banach's fixed point theorem, the uniqueness of the fixed point of $\mathcal{T}$, thus of the solution of the full system.
\end{proof}

\section{The high-field limit}

In this section we want to study a macroscopic description of the previous model, by means of a suitable hyperbolic limit.
As seen before, haptotaxis and chemotaxis are the key ingredients to describe the evolution of the system, so they are called to retain its influence  in the macroscopic description. Keeping these properties in the limit is a powerful motivation to study the hyperbolic limit in the previous system.
Some references about scaling limits are for example \cite{BBNS3,BBNS1,BBNS2}, for parabolic and hyperbolic limits in a generic system with single/multiple populations modeled by the KTAP, or \cite{HypPar} to understand the connections between parabolic and hyperbolic scales in general kinetic theory.

\subsection{Hyperbolic scaling}

In this section we perform the typical fluid description of a kinetic model by way of a macroscopic limit of hyperbolic type. We note that equations (\ref{eqnQ}) and (\ref{eqnL}) are already macroscopic, so the scale should not change it. In the sequel, the interaction frequencies $p_h$, $p_l$, and $p_c$ are considered to be constant (otherwise, the scaling does not make sense). First of all, we define the dimensionless (``hat"{}) variables:
$$ t:=\hat{t}\tau,\qquad x:=\hat{x}R,\qquad v:=\hat{v}s_2,\qquad y:=\hat{y}R_0, $$
$$ f(t,x,v,y):=\bar f \hat{f}(\hat{t}, \hat{x}, \hat{v}, \hat{y}), \ \ Q(t,x, \theta) := R_0 \hat{Q} (\hat{t}, \hat{x}, \theta), \ \  L(t,x) := R_0 \hat{L} (\hat{t}, \hat{x}),
$$
$$ p_k(t,x,v,y):= \bar p_k, (k=h,l,c), \quad G(t,Q,L):=\bar G \hat{G}(\hat{y}, \hat{Q}, \hat{L}), $$
$$\alpha_j(y):= \hat{\alpha} _j(\hat{y}), (j=1,2), \qquad T(v,v'):=\frac{1}{s_2^n} \hat{T}(\hat{v}, \hat{v}'),  $$
$$
\psi(v;v',\theta) = \frac{1}{R_0 s_2^n}\,  \hat{\psi} (\hat{v};\hat{v}',\theta), \qquad K(v,v'):=\frac{1}{s_2^n} \hat{K}(\hat{v}, \hat{v}'),
$$
where $\tau, R, \bar f, \bar p_k,$ and $\bar G $ are typical quantities of their respective variables. The new variables are defined in the sets
$$ \hat{V}:=\frac{1}{s_2}V, \qquad \hat{Y}:=\frac{1}{R_0}Y. $$
Our system then becomes:
$$
\frac{\partial \hat{f}}{\partial \hat{t}} + \frac{s_2 \tau}{R}\hat{v}\cdot\nabla_{\hat{x}}\hat{f} + \frac{\tau \bar G}{R_0} \nabla_{\hat{y}} \cdot (\hat{G} \hat{f}) = \bar p_h \tau \hat{\mathcal{H}}(\hat{f} ,\hat{Q}) + \bar p_l \tau \hat{\mathcal{L}}(\hat{f} ) + \bar p_c \tau \hat{\mathcal{C}}(\hat{f} , \hat{L}),
$$
\begin{align*}
\frac{\partial \hat{Q}}{\partial \hat{t}} = & \ -\tau R_0^2 s_2^n \bar f \kappa \left( \int _{\hat{V}} \int _{\hat{Y}} \left( 1 - \left| \theta \cdot \frac{\hat{v}}{|\hat{v}|} \right| \right) \hat{f} d\hat{v}d\hat{y} \right) \hat{Q} \\
 & - \tau R_0^3 s_2^n \bar f k_1 \hat{Q}\int _{\hat{V}} \int _{\hat{Y}} (1-\hat{y}_1-\hat{y}_2)\hat{f} d\hat{v}d\hat{y} + \tau R_0^2 s_2^n \bar f \frac{k_{-1}}{\mathbb{S}^{n-1}} \int _{\hat{V}} \int _{\hat{Y}} \hat{y}_1\hat{f} d\hat{v}d\hat{y},  \\
\frac{\partial \hat{L}}{\partial \hat{t}} = & \ \tau R_0^2 s_2^n \bar f \kappa \int _{\mathbb{S}^{n-1}} \hspace{-0.1cm} \left( \int _{\hat{V}} \! \int _{\hat{Y}} \! \left( 1 - \left| \theta \cdot \frac{\hat{v}}{|\hat{v}|} \right| \right) \hat{f} d\hat{v}d\hat{y} \right) \hat{Q} d\theta - \tau r_L \hat{L} + \frac{\tau}{R^2}D_L \Delta _{\hat{x}} \hat{L} \\
& - \tau R_0^3 s_2^n \bar f k_2 \hat{L}\int _{\hat{V}} \int _{\hat{Y}} (1-\hat{y}_1-\hat{y}_2)\hat{f} d\hat{v}d\hat{y} + \tau R_0^2 s_2^n \bar f k_{-2}\int _{\hat{V}} \int _{\hat{Y}} \hat{y}_2\hat{f} d\hat{v}d\hat{y}.
\end{align*}

We impose first the normalization restrictions $\frac{s_2 \tau}{R} =1$ and $\frac{\tau}{R^2}D_L=1$.
The hyperbolic scaling corresponds to the choice
$$ \tau \bar p_l = \frac{1}{\varepsilon }, $$
i.e., the turning time $\frac{1}{\bar p_l}$ is very small compared to the typical time $\tau $. 

There are three other phenomena (cell membrane reactions, haptotaxis and quemotaxis) to consider. We rescale the corresponding terms, assuming also that their frequencies are small compared to the turning frequency $\bar p_l$. More precisely, we choose the following relations:
$$ \frac{\bar G}{R_0} = \varepsilon ^{a} \bar p_l , \ \ \bar p_h  = \varepsilon ^{b} \bar p_l, \ \ \bar p_c= \varepsilon ^{d}\bar p_l, $$
where $0<a<1$, $b,d \geq 1$.

To scale the other two equations, we recall that they are actually  macroscopic, so they will preserve their form. This is why we only define the scaled non--dimensional constants involved:
$$
\hat{\kappa}:= \tau R_0^2 s_2^n \bar{f} \kappa  , \quad
\hat r_L:=\tau r_L, \quad
$$
$$
\hat{k}_i:= \tau R_0^3 s_2^n \bar{f} k_i , \quad
\hat{k}_{-i}:= \tau R_0^2 s_2^n \bar{f} k_{-i}, \quad (i=1,2).
$$
Skipping the ``hat"{} for the non--dimensional variables, our system becomes
\begin{align}
& \varepsilon \left(  \frac{\partial f_\varepsilon }{\partial t} + v\cdot \nabla _x f_\varepsilon \right) + \varepsilon ^a \nabla _y \cdot (G(y, Q_\varepsilon , L_\varepsilon ) f_\varepsilon ) =  \varepsilon ^b \mathcal{H}(f_\varepsilon ,Q_\varepsilon ) \nonumber \\
& + \mathcal{L}(f_\varepsilon ) + \varepsilon ^d \mathcal{C}(f_\varepsilon , L_\varepsilon ) \label{faprox}
\end{align}
for the cell population, where
\begin{align}
\mathcal{H}(f_\varepsilon,Q_\varepsilon) (t,x,v,y) & =\int_V \int_{\mathbb{S}^{n-1}} \psi(v; v', \theta)f_\varepsilon(t,x,v',y)Q_\varepsilon(t,x,\theta ) d\theta dv' \nonumber \\
 & - f_\varepsilon(t,x,v,y)\bar{Q}_\varepsilon(t,x); \label{Hesc} \\
\mathcal{L}(f_\varepsilon)(t,x,v,y) & =  \int_V \alpha_1(y) T(v,v')f_\varepsilon(t,x,v',y)dv' \nonumber \\
 & -\alpha_1(y) f_\varepsilon(t,x,v,y); \label{Lesc} \\
\mathcal{C} (f_\varepsilon,L_\varepsilon)(t,x,v,y) & =  \int_V \alpha_2(y) K[\nabla L_\varepsilon](v,v')f_\varepsilon(t,x,v',y)dv'\nonumber \\
& - \alpha_2(y) f_\varepsilon(t,x,v,y); \label{Cesc}
\end{align}

and
\begin{align}
\frac{\partial Q_\varepsilon }{\partial t} = & \ -\kappa \left( \int _V \int _Y \left( 1 - \left| \theta \cdot \frac{v}{|v|} \right| \right) f_\varepsilon dvdy \right) Q_\varepsilon \nonumber \\
 & -k_1 Q_\varepsilon \int _V \int _Y (1-y_1-y_2)f_\varepsilon dvdy + \frac{k_{-1}}{\mathbb{S}^{n-1}}\int _V \int _Y y_1f_\varepsilon dvdy, \label{Qaprox} \\
\frac{\partial L_\varepsilon}{\partial t} = & \ \kappa\int _{\mathbb{S}^{n-1}}\left( \int _V \int _Y \left( 1- \left| \theta \cdot \frac{v}{|v|} \right|  \right) f_\varepsilon dvdy \right) Q_\varepsilon d\theta - r_L L_\varepsilon + \Delta _x L_\varepsilon  \nonumber \\
& - k_2 L_\varepsilon \int _V \int _Y (1-y_1-y_2)f_\varepsilon dvdy + k_{-2}\int _V \int _Y y_2f_\varepsilon dvdy,
\label{Laprox}
\end{align}
for the chemicals, where
$$ V=[ s, 1] \times \mathbb{S}^{n-1} , \   Y=\{(y_1,y_2) \in (0,1)\times (0,1) \ : \  y_1 +y_2 <1\}, $$
and $s:=s_1/s_2 $.
\subsection{Limiting equations}
We first collect some standard assumptions on the turning operator $\mathcal{L}$:
\begin{itemize}
\item \textbf{Solvability conditions.} $\mathcal{L}$ satisfies $\displaystyle \int_V \mathcal{L}(f)dv=0,$ and $\displaystyle  \int_V v\mathcal{L}(f) dv =0$.
\item \textbf{Kernel of $\mathcal{L}$.} For all $\rho \geq 0, U\in \mathbb{R}^n$, there exists a unique function $M_{\rho, U}\in L^1(V\times Y,(1+|v|)dv+dy)$ such that
\begin{equation}
 \mathcal{L}(M)=0,\qquad \int _V\int _Y M dvdy =\rho , \qquad \int _V\int _Y vM dvdy = \rho U.
 \label{pp1}
\end{equation}
\end{itemize}
Let us formally deduce the limiting equations. To that aim, we first study the equations verified by the moments of $f_\varepsilon $:
$$
\rho _\varepsilon := \int _V\int _Y f_\varepsilon dvdy, \qquad \rho _\varepsilon U_\varepsilon  := \int _V\int _Y vf_\varepsilon dvdy, \qquad \rho _\varepsilon W_\varepsilon :=\int _V\int _Y yf_\varepsilon dvdy. $$
If we make $\varepsilon =0$ in (\ref{faprox}) we obtain  $\mathcal{L}(f_0)=0$, and therefore we deduce that the limiting function necessarily has the form  $f_0=M_{\rho _0, U_0}$. Let us now derive the macroscopic equations verified by the moments. As usual,  integrating Eq. (\ref{faprox}) in $v$ and $y$, we get mass conservation:
\begin{equation}
\frac{\partial \rho _\varepsilon }{\partial t} + \nabla_x \cdot (\rho _\varepsilon U_\varepsilon ) =0.
\label{ppmass}
\end{equation}
Then, multiplying (\ref{faprox}) by $v$ and again integrating, we obtain:
\begin{align}
&\frac{\partial (\rho _\varepsilon U_\varepsilon)}{\partial t} + \nabla_x \cdot (\mathbb{P}_\varepsilon + \rho _\varepsilon U_\varepsilon \otimes U_\varepsilon ) \nonumber \\ \label{ppmomu}
&=   \varepsilon^{b-1} \int _V\int _Y v\mathcal{H}(f_\varepsilon, Q_\varepsilon) dvdy   + \varepsilon^{d-1}\int _V\int _Y v\mathcal{C}(f_\varepsilon, L_\varepsilon) dvdy,
\end{align}
where $\mathbb{P} _\varepsilon$ is the pressure tensor
$$ \mathbb{P} _\varepsilon (t,x):=\int_V \int_Y (v-U_\varepsilon)\otimes (v-U_\varepsilon)f_\varepsilon dvdy.
$$
Finally, multiplying (\ref{faprox}) by $y$ and integrating, we obtain:
\begin{equation}
\varepsilon \partial _t{\rho _\varepsilon W_\varepsilon } + \varepsilon \nabla_x \cdot  \int _V \int_Y y\otimes v f dvdy + \varepsilon^a (A_\varepsilon
W_\varepsilon  -b_\varepsilon) \rho_\varepsilon = 0,
\label{ppmomy}
\end{equation}
where the matrix $A_\varepsilon$ and the vector $b_\varepsilon$ are respectively given by
$$  A_\varepsilon := \begin{pmatrix} k_1 \bar{Q}_\varepsilon + k_{-1} & k_1 \bar{Q}_\varepsilon \\ k_2 L_\varepsilon & k_2 L_\varepsilon + k_{-2} \end{pmatrix},    \quad
b_\varepsilon= \begin{pmatrix} k_1 \bar{Q_\varepsilon} \\ k_2 L_\varepsilon \end{pmatrix}.
$$
Now, we assume that the solutions are small perturbations of the limit, \\ $f_\varepsilon = M_{\rho _0, U_0} + \varepsilon f_1$, hence a  Hilbert expansion of $f_\varepsilon$ around $M_{\rho _0, U_0}$.
Inserting this profile into (\ref{ppmass}), (\ref{ppmomu}) and (\ref{ppmomy}), we obtain
\begin{align*}
& \frac{\partial \rho_0 }{\partial t} + \nabla_x \cdot (\rho_0 U_0) =  \  0, \\
&\frac{\partial (\rho _0 U_0)}{\partial t} + \nabla_x \cdot(\mathbb{P}_0 + \rho _0 U_0\otimes U_0 ) =   \varepsilon^{b-1} \int _V\int _Y v\mathcal{H}(M_{\rho _0, U_0}, Q_\varepsilon) dvdy \\
  &  \hspace{3 cm}+  \varepsilon^{d-1}\int _V\int _Y v\mathcal{C}(M_{\rho _0, U_0}, L_\varepsilon) dvdy +  \mathcal{O}(\varepsilon ^b)  + \mathcal{O}(\varepsilon ^d), \\
& (A_\varepsilon W_\varepsilon  -b_\varepsilon) \rho_0 =  \  \mathcal{O}(\varepsilon ^{1-a}).
\end{align*}

Formally,  denoting $Q_0$, $L_0$, $A_0$ and $b_0$  the respective limits of $Q_\varepsilon $, $L_\varepsilon$, $A_\varepsilon$ and $b_\varepsilon$ and taking the limit $\varepsilon \rightarrow 0$, we obtain:
\begin{align*}
& \frac{\partial \rho_0 }{\partial t} + \nabla_x \cdot (\rho_0 U_0) =  \  0, \\
&\frac{\partial (\rho _0 U_0)}{\partial t} + \nabla_x \cdot(\mathbb{P}_0 + \rho _0 U_0\otimes U_0 ) =  \delta_{b,1} \int _V\int _Y v\mathcal{H}(M_{\rho _0, U_0}, Q_0) dvdy \\
  &  \hspace{3 cm}+  \delta_{d,1}  \int _V\int _Y v\mathcal{C}(M_{\rho _0, U_0}, L_0) dvdy , \\
& (A_0 W_0  -b_0) \rho_0 =  0,
\end{align*}
where different regimes can be observed ($\delta_{i,j}$ stands for the Kronecker delta), depending on the choice of the parameters $b$ and $d$. Actually, if $b,d>1$, it is a  pure hyperbolic system; if $b>1$ and $d=1$, we get a system with an additional chemotactic term; for $b=1$ and $d>1$ the additional term entails haptotaxis; finally, for  $b=d=1$ we find the whole hyperbolic system including both phenomena. In all cases, the third equation stands for a linear system that can be solved explicitly, so obtaining the limiting distributions of compounds, i.e. the $y$--activity--moment of $M_{\rho_0,U_0}$  expressed as follows:
\begin{equation}
 \hspace{-0.2cm}\rho_0W_0 = \! \int _V \! \int _Y \! y M_{\rho_0,U_0} dy \, dv = \frac{\rho_0 }{k_1k_{-2} \bar{Q_0}  + k_{-1}k_2 L_0 + k_{-1}k_{-2}} \! \begin{pmatrix} k_1k_{-2}\bar{Q_0}  \\ k_{-1}k_2 L_0 \end{pmatrix}. \label{limY}
\end{equation}
On the other hand, limiting equations for $\{ Q_\varepsilon \}$ and $\{ L_\varepsilon \}$ can be deduced. For example, in the $\{ Q_\varepsilon \}$ equation we introduce the expansion for $f_\varepsilon $, and formally taking the limit $\varepsilon \to 0$,  we obtain:
\begin{align*}
\frac{\partial Q_0 }{\partial t} = & -\kappa \left( \int _V \int _Y \left(1- \left| \theta \cdot \frac{v}{|v|} \right| \right) M_{\rho_0, U_0} dvdy \right) Q_0  \nonumber \\
& -k_1 Q_0 \int _V \int _Y (1-y_1-y_2)M_{\rho_0, U_0} dvdy +\frac{k_{-1}}{|\mathbb{S}^{n-1}|} \int _V \int _Y y_1M_{\rho_0, U_0} dvdy.
\end{align*}
By using  now the expression for the compounds stated in (\ref{limY}), we get:
\begin{align*}
\frac{\partial Q_0 }{\partial t} = & \ -\kappa \left( \int _V \int _Y \left(1- \left| \theta \cdot \frac{v}{|v|} \right| \right) M_{\rho_0, U_0} dvdy \right) Q_0  \nonumber \\
& + \frac{k_1 k_{-1} k_{-2} \rho_0 }{k_1k_{-2} \bar{Q}_0  + k_{-1}k_2 L_0 + k_{-1}k_{-2}}\left( \frac{\bar{Q}_0 }{|\mathbb{S}^{n-1}|}  -Q_0 \right).
\end{align*}
Note here that the last term cancels when integrating with respect to $\theta$. Applied to $L_\epsilon$, the same argument yields
\begin{align*}
\frac{\partial L_0}{\partial t} = & \ \kappa \int _{\mathbb{S}^{n-1}}\left( \int _V \int _Y \left( 1 - \left| \theta \cdot \frac{v}{|v|} \right| \right) M_{\rho_0, U_0} dvdy \right) Q_0 d\theta - r_L L_0 + \Delta _x L_0  \nonumber \\
& -k_2 L_0 \int _V \int _Y (1-y_1-y_2)M_{\rho_0, U_0} dvdy + k_{-2}\int _V \int _Y y_2M_{\rho_0, U_0} dvdy,
\end{align*}
that after use of (\ref{limY})  becomes:
$$
\frac{\partial L_0}{\partial t} = \ \kappa \int _{\mathbb{S}^{n-1}}\left( \int _V \int _Y \left( 1 - \left| \theta \cdot \frac{v}{|v|} \right| \right) M_{\rho_0, U_0} dvdy \right) Q_0 d\theta - r_L L_0 +  \Delta _x L_0.
$$
The main result of this section is the following.

\begin{teorema} \label{teolim}
Let $\{ (f_\varepsilon, Q_\varepsilon, L_\varepsilon )\}$ be the solution of (\ref{faprox})-- (\ref{Qaprox})--(\ref{Laprox}) verifying the hypotheses  of Lemma \ref{IntOp} and
$$\{\| f_\varepsilon \| _{L^\infty (0,T; (L^1 \cap L^\infty)(\mathbb{R}^n \times V \times Y))} + \| Q_{0,\varepsilon } \| _{L^\infty(\mathbb{R} ^n \times \mathbb{S} ^{n-1} )} + \| L_{0,\varepsilon } \| _{L^\infty(\mathbb{R} ^n)} \} <C < \infty .$$
Assume also that  the sequence $\{ (f_\varepsilon, Q_\varepsilon, L_\varepsilon )\}$ converges a.e.
Then, the a.e. limit of $f_\varepsilon$ is the function $M_{\rho, U}$ given by the properties of $\mathcal{L}$, where $\rho$, $U$ and $W$ are the respective $L^1$--strong limits of $\rho _\varepsilon $, $ U_\varepsilon$ and $ W_\varepsilon$. Also, the sequences $ \{ Q_\varepsilon\} , \{L_\varepsilon \}$ converge $L^\infty$--weakly${}^\ast$  to some functions $Q, L$.

Moreover, $\rho, U,$ and $W$ solve the following system:
\begin{align*}
&\frac{\partial \rho }{\partial t} + \nabla_x \cdot(\rho U) =  0, \\
& \frac{\partial (\rho U)}{\partial t} + \nabla_x \cdot(\mathbb{P}_0 + \rho U\otimes U) = \delta_{b-1} \int _V\int _Y v\mathcal{H}(M_{\rho, U}, Q) dvdy \\
 & \hspace{3 cm}  + \delta_{d-1} \int _V\int _Y v\mathcal{C}(M_{\rho, U}, L) dvdy,\\
& \rho W = \frac{\rho}{k_1k_{-2} \bar{Q}  + k_{-1}k_2 L + k_{-1}k_{-2}} \begin{pmatrix} k_1k_{-2}\bar{Q}  \\ k_{-1}k_2 L \end{pmatrix}.
\end{align*}
Besides, $Q$ and $ L$ satisfy
\begin{align*}
\frac{\partial Q }{\partial t} = & \ -\kappa \left( \int _V \int _Y \left(1- \left| \theta \cdot \frac{v}{|v|} \right| \right) M_{\rho, U} dvdy \right) Q  \nonumber \\
& + \frac{k_1 k_{-1} k_{-2} \rho }{k_1k_{-2} \bar{Q}  + k_{-1}k_2 L + k_{-1}k_{-2}}\left( -Q + \frac{\bar{Q} }{|\mathbb{S}^{n-1}|} \right),\label{limQ}
\end{align*}
$$
\frac{\partial L}{\partial t} =  \ \kappa \int _{\mathbb{S}^{n-1}}\left( \int _V \int _Y \left( 1 - \left| \theta \cdot \frac{v}{|v|} \right| \right) M_{\rho, U} dvdy \right) Q d\theta - r_L L + \Delta _x L. \label{limL} $$
\end{teorema}

\begin{proof} We first observe that the variables $v$ and $y$ are defined on bounded sets. Then, it can be deduced that the sequences of moments
$ \{ \rho_\varepsilon \} $, $\{ \rho_\varepsilon U_\varepsilon \}$, and $\left\{ \rho_\varepsilon W_\varepsilon  \right\} $
are uniformly bounded in $L^\infty(0,T; L^\infty(\mathbb{R}^n))$. On the other hand, returning to the inequalities (\ref{desQ}) and (\ref{desL}) for $Q_\varepsilon$ and $L_\varepsilon$ with $h=0$, it follows that the sequences $\{ Q_\varepsilon \}$ and $\{ L_\varepsilon \}$ are uniformly bounded. So, we can pass to the limit, up to a subsequence, in the weak${}^\ast$ topology of $L^\infty $ in all of them.

Now, the  integral operators $\{ \mathcal{H}(f_\varepsilon, Q_\varepsilon) \}$, $\{ \mathcal{L}(f_\varepsilon ) \}$ and $\{ \mathcal{C}(f_\varepsilon, L_\varepsilon) \}$ can be also uniformly bounded in $L^\infty(0,T;L^\infty(\mathbb{R}^n \times V \times Y ) ) $, so that their convergence is guaranteed. Nonetheless, we have to identify the limits of these sequences which involve quadratic terms.

The weak${}^\ast$ limit of $\{ f_\varepsilon \}$ has to coincide with its pointwise limit, called $f_0$. Then, via the Dunford--Pettis theorem this convergence holds also weakly in $L^1([0,T]\times \mathbb{R}^n\times V \times Y)$ locally, and thus strongly. As consequence $ \mathcal{L}(f_\varepsilon) $ converges strongly to $\mathcal{L}(f_0) $ in the same space, too.

The shown convergences are enough to take limits,  at least in a distributional sense,  in the linear terms of the  equations. Now we have to take care of the nonlinear terms. More precisely, we only have to observe that the strong convergence in $L^1_{loc} ([0,T]\times \mathbb{R}^n\times V \times Y)$ of $f_\varepsilon$, combined with the weak${}^\ast$--$L^\infty $ convergence of $Q_\varepsilon $ and $L_\varepsilon$, produces the following convergence
$$ \mathcal{H}(f_\varepsilon, Q_\varepsilon ) \rightarrow \mathcal{H}(f_0, Q), \quad \mbox{and } \quad  \mathcal{C}(f_\varepsilon, L_\varepsilon ) \rightarrow \mathcal{C}(f_0, L),
$$
in a distributional sense. Analogously, due to the fact that $v$ and $y$ are defined in sets of finite measure, this convergence also holds for the moments
$$
\int_V \int_Y v \mathcal{H}(f_\varepsilon, Q_\varepsilon )dvdy, \qquad \int_V \int_Y v\mathcal{C}(f_\varepsilon, L_\varepsilon )dvdy.
$$
We can now take the distributional limit in Eq. (\ref{faprox}) to obtain $\mathcal{L}(f_0)=0$. Using the properties of $\mathcal{L}$, we deduce that $f_0=M_{\rho, U}$, with $\rho$ and $U$ as in (\ref{pp1}). Again, the variables $v$ and $y$ are defined in sets of finite measure, so the previous argument can be rewritten for the sequences of  moments, so that the following convergences are deduced:
$$ \rho_\varepsilon \rightarrow \rho, \quad  \rho_\varepsilon U_\varepsilon \rightarrow \rho U, \quad  \mathbb{P}_\varepsilon \rightarrow \mathbb{P}_0, \quad \mbox{and } \  \rho_\varepsilon W_\varepsilon \rightarrow \rho W. $$
Finally, we can pass to the limit in the macroscopic equations (\ref{ppmass}), (\ref{ppmomu}), (\ref{ppmomy}), (\ref{Qaprox}) and (\ref{Laprox}) to deduce the announced result.
\end{proof}

\subsection{A particular case}
There are two ``hidden"{} problems in the previous development: The first one is that the system of macroscopic equations is not closed. Actually, the pressure tensor $\mathbb{P}_0$ and the integral operators appearing in the macroscopic equation for $\rho U$, involve some integrals with respect to the microscopic state that have not been expressed as functions of the macroscopic variables. Solving this problem requires an explicit expression of the function $M_{\rho,U}$. Here we present a particular case, as done in \cite{BBNS1}, choosing a particular turning operator  $\cal L$ for which all the computations can be done.

Consider that $\alpha_1$ is a constant function and a kernel with the form
$$ T(v,v'):= \lambda + \beta v\cdot v' ,$$
with positive constants $\lambda $ and $\beta $  verifying
\begin{equation} \lambda =\beta \frac{1-s^{n+2}}{(1-s^n)(n+2)}. \label{relL} \end{equation}
Then, the operator $\mathcal{L}$ given in (\ref{Lesc}) can be written as follows:
\begin{equation}
\mathcal{L} (f) := \lambda \int _V f(t,x,v',y) dv' + \beta v \cdot \int _V v' f(t,x,v',y) dv' - \lambda |V|f(t,x,v,y). \label{nuevoL}
\end{equation}
Note that, following the definition of  $\cal L$, the free parameter $\lambda$ has to fulfill $\lambda |V|=\alpha_1$ .
Let us compile the properties of this operator.

\begin{lema}
Let $\mathcal{L}$ be given by (\ref{nuevoL}), with $\lambda $ and $\beta $ satisfying (\ref{relL}). Then, the following properties are fulfilled by $\mathcal{L}$:
\begin{itemize}
\item $\mathcal{L}$ verifies the hypothesis of Lemma \ref{IntOp}.
\item $\mathcal{L}$ verifies the hyperbolic solvability conditions: $\int \mathcal{L}(f)dv= \int v\mathcal{L}(f) dv =0$.
\item Given $\rho \geq 0$ and $ U\in \mathbb{R}^n$, the function $M_{\rho , U}$ in the kernel of $\cal L$ verifying (\ref{pp1}) is given by
\begin{equation} M_{\rho , U}:= \frac{\rho }{|V||Y|} \left( 1+\frac{\beta }{\lambda } v\cdot U \right). \label{MrhoU} \end{equation}
\end{itemize}
\end{lema}
\begin{proof}  Checking that $\mathcal{L}$ verifies the hypotheses of Lemma \ref{IntOp} is immediate.

The solvability conditions can be easily deduced after integrating expression (\ref{nuevoL}) and using
$$ |V| = \frac{|\mathbb{S}^{n-1}|}{n} (1-s^n), \ \ \int _V vdv =0, \ \ \int _V v_i v_k dv=\frac{\delta_{ik} \, |\mathbb{S}^{n-1}|}{n(n+2) }(1-s^{n+2}). $$
We just check the second one. Multiplying by $v$ and integrating, we find
\begin{eqnarray*}
\int_Y \int_V v \mathcal{L}(f)dvdy & = & \lambda \int_Y \int_V \int_V vf(t,x,v',y)dv'dvdy \\
 & & + \beta \int _Y \int _V \int _V (v\otimes v)v'f(t,x,v',y) dv'dvdy \\
 & & - |V|\lambda \int _Y \int _V vf(t,x,v,y) dvdy.
\end{eqnarray*}
The first term vanishes, via the Fubini theorem and $\int _V vdv =0 $. Analogously,
\begin{align*}
& \beta \int _Y \int _V \int _V (v\otimes v)v'f(t,x,v',y) dv'dvdy \\
& = \frac{|\mathbb{S}^{n-1}|}{n(n+2)}(1-s^{n+2}) \beta \int _Y \int _V vf(t,x,v,y) dvdy,
\end{align*}
so that
$$\int_Y \int_V v \mathcal{L}(f)dvdy = \left(\frac{|\mathbb{S}^{n-1}|}{n(n+2)}(1-s^{n+2}) \beta - |V|\lambda   \right) \int _Y \int _V vf(t,x,v,y) dvdy, $$
which vanishes once the value of $|V|$ and the relation (\ref{relL}) are considered.

Finally, it is straightforward to check that the function $M_{\rho,U}$ given by (\ref{MrhoU}) verifies  (\ref{pp1}), and then generates the kernel of $\cal L$.
\end{proof}

Let us now compute  the pressure tensor $\mathbb{P}_0$ and the integral operators (\ref{Hesc}) and (\ref{Cesc})  for this choice of the turning operator. First, we calculate:
\begin{eqnarray*}
\int _Y \int _V v\otimes v M_{\rho, U}dvdy & = & \frac{\rho }{|V||Y|}\int _Y \int _V v\otimes v dvdy \\
 & & + \frac{\rho }{|V||Y|}\frac{\beta}{\lambda} \int _Y \int _V (v\otimes v) v\cdot U dvdy \\
 & = & \rho \frac{1-s^{n+2}}{1-s^{n}}\frac{2}{(n+2)} {\mathbb I} + 0,
\end{eqnarray*}
where  ${\mathbb I}$ denotes the identity matrix. Then, we conclude that
\begin{eqnarray*}
\mathbb{P}_0  =  2\frac{1-s^{n+2}}{(n+2)(1-s^n)}\rho {\mathbb I} -\rho U\otimes U.
\end{eqnarray*}

Now, we can compute the $v$-moments of the integral operators. Define first the following macroscopic quantities:
\begin{align*}
\Psi^1(\theta ):=&  \int_V\int_V v \psi(v; v',\theta )dvdv', \\
\Psi^2(\theta ):=&  \int_V\int_V v\otimes v' \psi(v; v',\theta )dvdv',\\
\mathcal{K}^1[L]:=&   \int_V\int_V v K[\nabla L](v, v')dvdv',\\
\mathcal{K}^2[L]:=&   \int_V\int_V v\otimes v'  K[\nabla L](v, v')dvdv',
\end{align*}
and construct the macroscopic integral operators:
\begin{align*}
H(\rho, U, Q):= & \frac{\rho }{|V|} \left( \int \left( \Psi^1(\theta)+\frac{\beta}{\lambda}\Psi^2(\theta) \cdot U \right) Q(\theta)d\theta - \bar{Q}|V| U \right), \\
C(\rho, U, L):= & \frac{\rho \, \alpha_2 }{|V|} \left( \mathcal{K}^1[L]+\left( \frac{\beta}{\lambda}\mathcal{K}^2[L] - |V| \mathbb{I} \right) \cdot U \right).
\end{align*}
Using the expression (\ref{MrhoU}) in (\ref{Hesc}) and (\ref{Cesc}), a simple calculation leads to:
\begin{align*}
\int_V \int_Y v\mathcal{H}(M_{\rho, U}, Q) dvdy = & H(\rho, U, Q), \\
\int_V \int_Y v\mathcal{C}(M_{\rho, U}, L) dvdy = & C(\rho, U, L).
\end{align*}

Finally, Theorem \ref{teolim} can be  rewritten for this particular case as follows.

\begin{teorema} 
Under the hypotheses of Theorem \ref{teolim} and with the turning operator given by (\ref{nuevoL}), the limiting equations for $\rho $ and $\rho U$ are the following:
\begin{align*}
& \frac{\partial \rho }{\partial t} + \nabla_x \cdot(\rho U) = 0, \\
& \frac{\partial (\rho U)}{\partial t} + 2\frac{1-s^{n+2}}{(n+2)(1-s^n)}\nabla _x \rho =  \delta_{b-1} H(\rho, U, Q) + \delta_{d-1} C(\rho, U, L),\\
& \rho W = \frac{\rho}{k_1k_{-2} \bar{Q}  + k_{-1}k_2 L + k_{-1}k_{-2}} \begin{pmatrix} k_1k_{-2}\bar{Q}  \\ k_{-1}k_2 L \end{pmatrix}.
\end{align*}
Also, the limiting equations for $Q$ and $L$ can be written as follows:
$$ \frac{\partial Q }{\partial t} = \ -\kappa \frac{\rho }{|V|}g(\theta) Q+ \frac{k_1 k_{-1} k_{-2} \rho }{k_1k_{-2} \bar{Q}  + k_{-1}k_2 L + k_{-1}k_{-2}}\left( -Q + \frac{\bar{Q} }{|\mathbb{S}^{n-1}|} \right), $$
$$
\frac{\partial L}{\partial t} = \ \kappa \frac{\rho }{|V|} \int _{\mathbb{S}^{n-1}}g(\theta)Q(\theta ) d\theta - r_L L + \Delta _x L,
$$
where $g(\theta )=\displaystyle \int _V (1-\Big|\frac{\theta \cdot v}{|v|}\Big|) dv. $
\end{teorema}

\section*{Acknowledgements}
Partially supported by the Spanish Government, project MTM2011--23384 and Junta de Andaluc\'ia Project FQM-954.

\end{document}